\documentclass[12pt]{article}
\usepackage{amssymb,amsfonts,amsmath,amsrefs,amsthm,hyperref,bbm,dsfont,tikz,wrapfig,tikz-cd}
\usepackage{geometry}
\newcommand{\8}{\infty}

\newcommand{\C}{\mathbb{C}}
\newcommand{\N}{\mathbb{N}}

\newcommand{\T}{\mathbb{T}}
\newcommand{\D}{\mathbb{D}}

\newcommand{\spa}{\mathrm{span}}
\newcommand{\Ker}{\mathrm{Ker}}

\newcommand{\supp}{\mathrm{supp}}
\newcommand{\Co}{\mathcal{C}}
\newcommand{\Fo}{\mathcal{F}}

\newcommand{\Lo}{\mathcal{L}}

\newcommand{\Ho}{\mathcal{H}}
\newcommand{\Int}{\mathrm{int}}

\newcommand{\1}{\mathds{1}}

\newcounter{erz}[section] \numberwithin{erz}{section}
\newtheorem{theorem}[erz]{Theorem}
\newtheorem{lemma}[erz]{Lemma}
\newtheorem{proposition}[erz]{Proposition}

\newtheorem{corollary}[erz]{Corollary}
\newtheorem{question}[erz]{Question}
\theoremstyle{remark}
\newtheorem{remark}[erz]{Remark}
\newtheorem{example}[erz]{Example}

\begin{document}
\title{Maximum Modulus Principle For Multipliers and Mean Ergodic Multiplication Operators}
\author{Eugene Bilokopytov\footnote{Email address bilokopi@myumanitoba.ca, erz888@gmail.com.}}
\maketitle

\begin{abstract}
The main goal of this note is to show that (not necessarily holomorphic) multipliers of a wide class of normed spaces of continuous functions over a connected Hausdorff topological space cannot attain their multiplier norms, unless they are constants. As an application, a contractive multiplication operator is either a multiplication with a constant, or is completely non-unitary. Additionally, we explore possibilities for a multiplication operator to be (weakly) compact and (uniformly) mean ergodic.


MSC2020 46B20, 47A35, 47B38
\end{abstract}

\section{Introduction}

Normed spaces of functions are ubiquitous in mathematics, especially in analysis. These spaces can be of a various nature and exhibit different types of behavior, and in this paper we discuss some questions related to these spaces from a general, axiomatic viewpoint. One of the most prominent classes of linear operators on the spaces of functions is the class of multiplication operators (MO). In this article we continue our investigation (see \cite{erz,erz2}) of the general framework which allows to consider any Banach space that consists of continuous (scalar-valued) functions, such that the point evaluations are continuous linear functionals, and of MO's on these spaces.\medskip

First, let us define precisely what we mean by a normed space of continuous functions. Let $X$ be a topological space (a \emph{phase space}) and let $\Co\left(X\right)$ denote the space of all continuous complex-valued functions over $X$ endowed with the compact-open topology. A \emph{normed space of continuous functions} (NSCF) over $X$ is a linear subspace $\mathbf{F}\subset\Co\left(X\right)$ equipped with a norm that induces a topology, which is stronger than the compact-open topology, i.e. the inclusion operator $J_{\mathbf{F}}:\mathbf{F}\to\Co\left(X\right)$ is continuous, or equivalently the unit ball $B_{\mathbf{F}}$ is bounded in $\Co\left(X\right)$. If $\mathbf{F}$ is a linear subspace of $\Co\left(X\right)$, then the \emph{point evaluation} at $x\in X$ on $\mathbf{F}$ is the linear functional $x_{\mathbf{F}}:\mathbf{F}\to\C$, defined by $x_{\mathbf{F}}\left(f\right)=f\left(x\right)$. If $\mathbf{F}$ is a NSCF, then all point evaluations are bounded on $\mathbf{F}$. Conversely, if $\mathbf{F}\subset\Co\left(X\right)$ is equipped with a complete norm such that $x_{\mathbf{F}}\in \mathbf{F}^{*}$, for every $x\in X$, then $\mathbf{F}$ is a NSCF. We will call a NSCF $\mathbf{F}$ over $X$ \emph{(weakly) compactly embedded} if $J_{\mathbf{F}}$ is a (weakly) compact operator, or equivalently, if $B_{\mathbf{F}}$ is (weakly) relatively compact in $\Co\left(X\right)$. Clearly, every compactly embedded NSCF's is weakly compactly embedded. On the other hand, any reflexive NSCF is also weakly compactly embedded. By a Banach space of continuous functions (BSCF) we mean a complete NSCF.\medskip

A \emph{multiplication operator} (MO) with \emph{symbol} $\omega:X\to\C $ is a linear map $M_{\omega}$ on the space $\Fo\left(X\right)$ of all complex-valued functions on $X$ defined by $$\left[M_{\omega}f\right]\left(x\right)=\omega\left(x\right)f\left(x\right),$$ for $x\in X$. Let $\mathbf{F}$ and $\mathbf{E}$ be NSCF's over $X$. If $M_{\omega}\mathbf{F}\subset\mathbf{E} $, then we say that $M_{\omega}$ is a multiplication operator from $\mathbf{F}$ into $\mathbf{E}$ (we use the same notation $M_{\omega}$ for what is in fact $\left.M_{\omega}\right|_{\mathbf{F}}$). If in this case $\mathbf{F}=\mathbf{E}$, then we will call $\omega$ a \emph{multiplier} of $\mathbf{F}$. If both $\mathbf{F}$ and $\mathbf{E}$ are BSCF's, then any MO between these spaces is automatically continuous due to Closed Graph theorem. However, in concrete cases it can be very difficult to determine all MO's between a given pair of NSCF's, and in particular to characterize all multipliers of a NSCF (see e.g. \cite{ggm} and \cite{vukotic}, where the multiplier algebras of some specific families of NSCF's are described).\medskip

In this article we explore how certain properties of MO's are reflected on their symbols. In \cite{erz2} we focused on MO's which are surjective isometries and isometries in general. In particular, we proved that on a wide class of NSCF's every isometric MO with a non-constant symbol is completely non-unitary, i.e. has no unitary restrictions. The methods employed there were of Banach space geometry. In the present article we attack the same problem using a topological approach. This allows to extend the mentioned result to contractions, and also somewhat enlarge the class of admissible NSCF's (see Corollary \ref{cnu}). In the process we discover a result reminiscent of the Maximum Modulus Principle from the Complex Analysis (see theorems \ref{mp} and \ref{either}).

In addition to contractive MO's we also consider the MO's which are (weakly) compact and (uniformly) mean ergodic. In particular, we show that a MO can be compact only in degenerate cases (see Proposition \ref{moc}) and show that a NSCF that has ``a lot'' of weakly compact MO's has to be weakly compactly embedded itself (Proposition \ref{mow}). We also abstract some of the results from \cite{bjr} and \cite{br} about mean ergodic MO's (see Theorem \ref{mem}, Proposition \ref{un} and Corollary \ref{unc}). This involves proving one more version of ``Maximum Modulus Principle'' for multipliers. We conclude with an alternative proof of the result from \cite{bjr} saying that a MO on a weighted Banach space of continuous functions is mean ergodic if and only if it is uniformly mean ergodic.\medskip

Throughout the paper by $Id_{X}$ we mean the identity map on a set $X$, while the supremum norm of $f:X\to\C$ is denoted by $\|f\|_{\8}$.\medskip

Before concluding this section with some concrete examples of NSCF's, let us mention a large class of compactly embedded NSCF's. If $X$ is a domain in $\C^{n}$, i.e. an open connected set, and $\mathbf{F}$ is a NSCF over $X$ that consists of holomorphic functions, then $\mathbf{F}$ is compactly embedded. Indeed, by Montel's theorem (see \cite[Theorem 1.4.31]{scheidemann}), $B_{\mathbf{F}}$ is relatively compact in $\Co\left(X\right)$, since it is a bounded set that consists of holomorphic functions.

\begin{example}
The Hardy space $\mathbf{H}$ is the BSCF over the (open) unit disk $\D\subset\C$ that consists of holomorphic functions $f$ with the norm defined by $\|f\|^{2}=\sum\limits_{n=0}^{\8}\left|a_{n}\right|^{2}$, where $\left\{a_{n}\right\}_{n=0}^{\8}$ are the Taylor coefficients of $f$. One can show that this is a Hilbert space with monomials forming an orthonormal basis. The Hardy space is among the most studied function spaces, and we refer to e.g. \cite{koosis} for more information.
\qed\end{example}\smallskip

\begin{example}
Assume that $X$ is a Tychonoff space and let $u:X\to\left(0,+\8\right)$ be upper semi-continuous. Define the \emph{weighted space of continuous functions} \linebreak $\Co_{u}^{\8}=\left\{f\in\Co\left(X\right),~\|f\|_{u}^{\8}=\|uf\|_{\8}<+\8\right\}$. One can show that this is a BSCF over $X$ with respect to the norm $\|\cdot\|_{u}^{\8}$. Moreover, $\|x_{\Co_{u}^{\8}}\|=\frac{1}{u\left(x\right)}$, for every $x\in X$. Indeed, it is immediate that $\left|f\left(x\right)\right|\le \frac{1}{u\left(x\right)}$, for every $f\in B_{\Co_{u}^{\8}}$, while since $\frac{1}{u}$ is lower semi-continuous and positive-valued, for any $\varepsilon>0$ there is a continuous $f:X\to\left(0,+\8\right)$ such that $f\le \frac{1}{u}$, while $f\left(x\right)\ge \frac{1}{u\left(x\right)}-\varepsilon$ (see \cite[IX.1.6, Proposition 5]{bourb}). One can also show that $\Co_{u}^{\8}$ is not weakly compactly embedded (the proof from [Example 2.6]\cite{erz2} carries over to the case when $u$ is non-constant). In the case $u\equiv 1$ we use the notation $\Co_{\8}\left(X\right)$. Additionally, $\Co_{u}^{0}$ consists of $f\in\Co_{u}^{\8}$ such $\left|uf\right|$ vanish at infinity.
\qed\end{example}\smallskip

\begin{example}
Assume that $X$ is a domain in $\C^{n}$ and let $u:X\to\left(0,+\8\right)$ be upper semi-continuous. Define the \emph{weighted space of holomorphic functions}  $\Ho_{u}^{\8}=\Co_{u}^{\8}\cap\Ho\left(X\right)$, where $\Ho\left(X\right)$ is the subspace of $\Co\left(X\right)$ that consists of holomorphic functions. As was mentioned before, this space is compactly embedded, and is a closed subspace of $\Co_{u}^{\8}$. Note however, that the equality $\|x_{\Co_{u}^{\8}}\|=\frac{1}{u\left(x\right)}$ may not hold (see \cite{bbt}).  In the case $u\equiv 1$ we use the notation $\Ho_{\8}\left(X\right)$. Additionally, $\Ho_{u}^{0}=\Ho_{u}^{\8}\cap\Co_{u}^{0}$.
\qed\end{example}\smallskip

\begin{example}\label{ad}
Assume that $X$ is a bounded domain in $\C^{n}$. Then $A\left(X\right)$ is the closed subalgebra of $\Ho_{\8}\left(X\right)$ which consists of functions that admit a continuous extension on $\overline{X}$. Another natural way to represent this space is $A\left(\overline{X}\right)$ which is the closed subalgebra of $\Co_{\8}\left(\overline{X}\right)$ that consists of functions holomorphic on $X$. Then $A\left(X\right)$ is a BSCF over $X$, and $A\left(\overline{X}\right)$ is a BSCF over $\overline{X}$. Both contain constant functions. Note that there can be no ambiguity in our notations since $X$ is open and $\overline{X}$ is compact. Also, note that MO's are same on these BSCF's, despite being induced by functions on $X$ and $\overline{X}$ respectively.\qed\end{example}

\section{General Properties of Multiplication Operators}

In this section we consider some basic properties of MO's, including the Maximal Modulus Principle for multipliers. Everywhere in this section $X$ is a Hausdorff space. Note that $M_{\omega}$ is continuous on $\Fo\left(X\right)$, for every $\omega:X\to\C$, and is continuous on $\Co\left(X\right)$ whenever $\omega$ is continuous. We will however mostly deal with MO's on NSCF's. These can be characterized by the following well-known fact (see e.g. \cite[Proposition 2.4 and Corollary 2.5]{erz}).

\begin{proposition}\label{rec}
Let $\mathbf{F}$ and $\mathbf{E}$ be NSCF's over $X$ and let $T\in\Lo\left(\mathbf{F},\mathbf{E}\right)$. Then $T=M_{\omega}$, for $\omega:X\to\C $, if and only if $T^{*}x_{\mathbf{E}}=\omega\left(x\right)x_{\mathbf{F}}$, for every $x\in X$. In other words, $T$ is a MO if and only if $T^{*}x_{\mathbf{E}}\subset\C x_{\mathbf{F}}$, for every $x\in X$.
\end{proposition}

If in particular $\mathbf{F}=\mathbf{E}$, then $T$ is a MO if and only if $x_{\mathbf{F}}$ is an eigenvector of $T^{*}$ (or else $x_{\mathbf{F}}=0_{\mathbf{F}^{*}}$), for every $x\in X$. Then the multiplier is the correspondence between $x$ and the eigenvalue of $T^{*}$ for $x_{\mathbf{F}}$.

For $Y\subset X$ define $\mathbf{F}_{Y}=\left\{f\in \mathbf{F}\left|\supp f\subset Y\right.\right\}=\left\{x_{\mathbf{F}}, x\in X\backslash Y\right\}^{\perp}$, which is a closed subspace of $\mathbf{F}$, and is also a NSCF over $X$. It follows that $\Ker M_{\omega}=\mathbf{F}_{ \omega^{-1} \left(0\right) }$, $\overline{M_{\omega}\mathbf{F}}\subset\mathbf{F}_{X\backslash \omega^{-1} \left(0\right) }$ and $M_{\omega}\mathbf{F}_{Y}\subset\mathbf{F}_{Y}$, for any $Y\subset X$.\medskip

Note that in general we cannot reconstruct the symbol of a MO from its data as a linear operator between certain NSCF's, in the sense that the equality of MO's does not imply the equality of their symbols. In order to prevent this pathology from happening we have to introduce the following concept. We will call a NSCF $\mathbf{F}$ over $X$ $1$-\emph{independent} if for every $x\in X$ we have $x_{\mathbf{F}}\ne 0_{\mathbf{F}^{*}}$, i.e. there is $f\in \mathbf{F}$ such that $f\left(x\right)\ne 0$. It is easy to see that a MO from a $1$-independent NSCF determines its symbol. Moreover, some properties of this symbol can also be recovered.

\begin{proposition}\label{hc} If $\mathbf{F}$ is a $1$-independent NSCF over $X$, then its multipliers are continuous. Moreover, $\|\omega\|_{\8}\le\|M_{\omega}\|$, for every multiplier $\omega$ of $\mathbf{F}$.
\end{proposition}
\begin{proof}
Let $x\in X$. Since $\mathbf{F}$ is $1$-independent, there is $f\in \mathbf{F}$ such that $f\left(x\right)\ne 0$. As $f$ and $M_{\omega}f$ are continuous, $\omega$ is continuous at $x$ as a ratio of continuous functions. Moreover, if $M_{\omega}\in\Lo\left(\mathbf{F}\right)$, then $M_{\omega}^{*}x_{\mathbf{F}}=\omega\left(x\right)x_{\mathbf{F}}$ implies that $\left|\omega\left(x\right)\right|\le\|M_{\omega}^{*}\|=\|M_{\omega}\|$. Since $x$ was chosen arbitrarily, the result follows.
\end{proof}

Let us now proceed to the main result of the paper. We need the following lemma.

\begin{lemma}\label{fp}
Let $\omega:X\to\C$ be such that for every $x\in \omega^{-1}\left(1\right)$ there is $f\in \Co\left(X\right)$ with $f\left(x\right)\ne 0$ and $\omega f = f$. Then $\omega^{-1}\left(1\right)$ is open.
\end{lemma}
\begin{proof}
Let $x\in \omega^{-1}\left(1\right)$ and let $f\in \Co\left(X\right)$ be such that $f\left(x\right)\ne 0$ and $\omega f = f$. The latter means that $f\left(y\right)\left(\omega\left(y\right)-1\right)=0$, for every $y\in X$. Since $f\left(x\right)\ne 0$ there is an open subset $U$ of $X$ containing $x$ such that $f\left(y\right)\ne0$, for any $y\in U$. Hence, $\omega\left(y\right)-1=0$, for every $y\in U$, from where $U\subset \omega^{-1}\left(1\right)$, and so $x\in\Int~ \omega^{-1}\left(1\right)$. Since $x$ was chosen arbitrarily, we conclude that $\omega^{-1}\left(1\right)$ is open.
\end{proof}

Recall that a subset $Y\subset X$ is called \emph{clopen} if it simultaneously closed and open. Let $\T=\partial \D$ be the unit circle.

\begin{theorem}\label{mp} Let $\mathbf{F}$ be a $1$-independent NSCF over $X$. If $\omega:X\to\C$ is such that $M_{\omega}$ is an operator of norm $1$ on $\mathbf{F}$, and moreover $M_{\omega}B_{\mathbf{F}}$ is relatively weakly compact in $\Co\left(X\right)$, then $\omega^{-1}\left(\lambda\right)$ is clopen, for every $\lambda\in\T$. Furthermore, if additionally $X$ is connected and there is $x\in X$ such that $\left|\omega\left(x\right)\right|=1$, then $\omega\equiv\lambda$, for some $\lambda\in\T$.
\end{theorem}
\begin{proof}
Let $\lambda\in\T$. Replacing $\omega$ with $\lambda\omega$ if needed we may assume that $\lambda =1$. Since $\omega$ is continuous due to Proposition \ref{hc}, $\omega^{-1}\left(1\right)$ is closed. Let $x\in \omega^{-1}\left(1\right)$.

Since $M_{\omega}B_{\mathbf{F}}$ is relatively weakly compact in $\Co\left(X\right)$, it follows (see \cite[4.3, Corollary 2]{floret}) that $B=\overline{M_{\omega}\overline{B}_{\mathbf{F}}}^{\Co\left(X\right)}$ is a convex pointwise compact set in $\Co\left(X\right)$. Since $\|M_{\omega}\|\le 1$ we have that $M_{\omega}\overline{B}_{\mathbf{F}}\subset \overline{B}_{\mathbf{F}}$, from where $M_{\omega}M_{\omega}\overline{B}_{\mathbf{F}}\subset M_{\omega}\overline{B}_{\mathbf{F}}$, and since $M_{\omega}$ is continuous on $\Co\left(X\right)$ we get $M_{\omega}B\subset B$. Due to pointwise compactness of $B$, the functional $x_{\mathbf{F}}$ attains its maximum on it, which is equal to $\sup\limits_{f\in \overline{B}_{\mathbf{F}} }\left|\left[M_{\omega}f\right]\left(x\right)\right|=\|M_{\omega}^{*}x_{\mathbf{F}}\|=\|\omega\left(x\right)x_{\mathbf{F}}\|=\|x_{\mathbf{F}}\|>0$.

Consider $D=\left\{f\in B\left|f\left(x\right)=\|x_{\mathbf{F}}\|\right.\right\}$, which is thus a non-empty convex pointwise compact set in $\Co\left(X\right)$. For every $f\in D$ we have $\left[M_{\omega}f\right]\left(x\right)=\omega\left(x\right)f\left(x\right)=\|x_{\mathbf{F}}\|$. Since $M_{\omega}B\subset B$, this implies that $M_{\omega}D\subset D$. Therefore, $\left.M_{\omega}\right|_{D}$ is a pointwise continuous self-map of a pointwise compact convex set $D$. Hence, by Tychonoff Fixed Point Theorem (see \cite[Theorem 12.19]{fhhmz}), there is $f\in D$ such that $M_{\omega}f=f$, i.e. $\omega f =f$.

Since $x\in \omega^{-1}\left(1\right)$ was chosen arbitrarily, from Lemma \ref{fp}, $\omega^{-1}\left(1\right)$ is open, and therefore clopen.\medskip

If $X$ is connected, every clopen set is either empty, or equal to $X$. Hence, if $\omega\left(x\right)=\lambda$, for some $\lambda\in\T$, the set $\omega^{-1}\left(\lambda\right)$ is clopen and nonempty, and so $\omega\equiv\lambda$.
\end{proof}

Let $J_{\mathbf{F}}$ be the embedding of $\mathbf{F}$ into $\Co\left(X\right)$. The condition of relative weak compactness of $M_{\omega}B_{\mathbf{F}}$ in $\Co\left(X\right)$ is equivalent to weak compactness of $J_{\mathbf{F}}M_{\omega}$ as an operator from $\mathbf{F}$ into $\Co\left(X\right)$. This operator is weakly compact whenever either $J_{\mathbf{F}}$ or $M_{\omega}$ is. In particular, we get the following ``Maximum Modulus Principle for Multipliers''.

\begin{theorem}\label{either}
Let $X$ be connected and let $\mathbf{F}$ be a $1$-independent NSCF over $X$. Let $\omega:X\to\C$ be a non-constant multiplier of $\mathbf{F}$. Assume furthermore that either $\mathbf{F}$ is weakly compactly embedded, or $M_{\omega}$ is a weakly compact operator on $\mathbf{F}$. Then $\left|\omega\left(x\right)\right|<\|M_{\omega}\|$, for every $x\in X$.
\end{theorem}

\begin{corollary}{[cf. \cite[Theorem 3.17]{erz2}]}\label{cnu}
Let $X$ be connected and let $\mathbf{F}$ be a $1$-independent weakly compactly embedded NSCF over $X$. If $\omega:X\to\C$ is a non-constant function such that $\|M_{\omega}\|\le 1$ on $\mathbf{F}$, then it is completely non-unitary, i.e. has no unitary restrictions to proper subspaces.
\end{corollary}
\begin{proof}
Let $\mathbf{E}$ be a non-trivial subspace of $\mathbf{F}$ (and so a NSCF) such that the restriction of $M_{\omega}$ on $\mathbf{E}$ is unitary. Let $x$ be such that $x_{\mathbf{E}}\ne 0_{\mathbf{E}^{*}}$. Then $\left(\left.M_{\omega}\right|_{\mathbf{E}}\right)^{*}$ is an isometry, and so $0\ne \|x_{\mathbf{E}}\|=\|\left(\left.M_{\omega}\right|_{\mathbf{E}}\right)^{*}x_{\mathbf{E}}\|=\|\omega\left(x\right)x_{\mathbf{E}}\|=\left|\omega\left(x\right)\right|\|x_{\mathbf{E}}\|$, from where $\left|\omega\left(x\right)\right|=1=\|M_{\omega}\|$. Contradiction.
\end{proof}

As an application we can conclude that non-trivial isometric multiplication operators have a Wold decomposition (see \cite{romanov}) into a direct sum of shifts.

\begin{example}
The classic example of a MO acting like a shift on a space of functions is the multiplication with a free constant on the Hardy space. Namely, let $\mathbf{H}$ be the Hardy space over $\D$, and let $\omega=Id_{\D}$. Then it is easy to see that $M_{\omega}$ is an isometry from $\mathbf{H}$ onto the $1$-co-dimensional subspace $\mathbf{H}_{\D\backslash\left\{0\right\}}$. Also, one can show that $M_{\omega}$ is an isometry on $\mathbf{H}$ if $\omega$ is an infinite Blashke product (see \cite[IV.A]{koosis}). In this case $M_{\omega}\mathbf{H}\subset\mathbf{H}_{\D\backslash \omega^{-1}\left(0\right)}$, which is of infinite co-dimension, since $\omega^{-1} \left(0\right)$ is infinite, and point evaluations are linearly independent on $\mathbf{H}$. Therefore, $M_{\omega}$ is an infinite sum of unit shifts.\qed
\end{example}

As was mentioned above, the condition of relative weak compactness of $M_{\omega}B_{\mathbf{F}}$ is always satisfied if $\mathbf{F}$ is weakly compactly embedded. It turns out that this implication can be reversed in some sense.

\begin{proposition}\label{mow}
A NSCF $\mathbf{F}$ over $X$ is weakly compactly embedded if and only if for every $x\in X$ there is a multiplier $\omega$ of $\mathbf{F}$ such that $\omega\left(x\right)\ne 0$ and $M_{\omega}B_{\mathbf{F}}$ is relatively weakly compact in $\Co\left(X\right)$.
\end{proposition}
\begin{proof}
Necessity: If $\mathbf{F}$ is weakly compactly embedded, $\omega\equiv 1$ is a multiplier of $\mathbf{F}$ that does not vanish and such that $M_{\omega}B_{\mathbf{F}}=B_{\mathbf{F}}$ is relatively weakly compact in $\Co\left(X\right)$.\medskip

Sufficiency: Let $J$ be the embedding of $\mathbf{F}$ into $\Fo\left(X\right)$. Let $x\in X$, and let $\omega$ be a multiplier of $\mathbf{F}$ such that $\omega\left(x\right)\ne 0$ and $B=M_{\omega}B_{\mathbf{F}}$ is relatively weakly compact in $\Co\left(X\right)$. Let $M$ be the multiplication operator with symbol $\omega$ defined on $\Fo\left(X\right)$ (we keep labeling the corresponding operator on $\mathbf{F}$ by $M_{\omega}$).

Since $\Fo\left(X\right)$ is reflexive we have $M^{**}=M$. Moreover, if  $T=JM_{\omega}=MJ$, then $MJ^{**}\overline{B}_{\mathbf{F}^{**}}=T^{**}\overline{B}_{\mathbf{F}^{**}}=\overline{TB_{\mathbf{F}}}$ (the proof of \cite[VI.4, Theorem 2]{ds} carries over to the case when the target space is locally convex). The latter set is in fact equal $\overline{B}^{\Fo\left(X\right)}$, and since $B$ is weakly compact in $\Co\left(X\right)$, it follows that $\overline{B}^{\Co\left(X\right)}$ is pointwise compact, from where $MJ^{**}\overline{B}_{\mathbf{F}^{**}}=\overline{B}^{\Fo\left(X\right)}=\overline{B}^{\Co\left(X\right)}\subset \Co\left(X\right)$.

Hence, if $f\in \overline{B}_{\mathbf{F}^{**}}$, then $g=J^{**}f$ is a function on $X$ such that $\omega g $ is continuous. Since $\omega$ is continuous and $\omega\left(x\right)\ne 0$, $g $ is continuous at $x$. Since $x$ and $f$ were chosen arbitrarily, it follows that $J^{**}\overline{B}_{\mathbf{F}^{**}}=\overline{B_{\mathbf{F}}}^{\Fo\left(X\right)}\subset \Co\left(X\right)$, and so $B_{\mathbf{F}}$ is relatively pointwise compact in $\Co\left(X\right)$ and so is relatively weakly compact in  (see \cite[4.3, Corollary 2]{floret}).
\end{proof}

Note that both of the conditions in the proposition depend on the way $\mathbf{F}$ sits in $\Co\left(X\right)$. It is natural to wonder if one can obtain an equivalence between stronger conditions intrinsic for $\mathbf{F}$:

\begin{question}
If $X$ is connected, is it true that a BSCF $\mathbf{F}$ over $X$ is reflexive if and only if for every $x\in X$ there is $\omega:X\to\C$ such that $\omega\left(x\right)\ne 0$ and $M_{\omega}$ is a weakly compact operator on $\mathbf{F}$?
\end{question}

Of course, necessity is obvious, since the $Id_{\mathbf{F}}$ is weakly compact as soon as $\mathbf{F}$ is reflexive. Without the assumption that $X$ is connected the answer is negative: consider $\mathbf{F}=l^{1}$ as a non-reflexive NSCF over $\N$, and then $\omega:\N\to\C$ defined by $\omega\left(n\right)=\frac{1}{n}$ does not vanish and generates a compact multiplication operator. Strong evidences in favour of the affirmative answer to the question are Corollary \ref{dp} below and Remark \ref{wme}.\medskip

While being weakly compact is a rather common property of multiplication operators (for example this is the case for all MO's if $\mathbf{F}$ is reflexive), compactness occurs only in trivial cases.

\begin{proposition}\label{moc}
Let $\mathbf{F}$ be a $1$-independent NSCF over $X$. If $\omega:X\to\C$ is such that $M_{\omega}$ is a compact operator on $\mathbf{F}$, then $\omega\left(X\right)$ is either finite, or is a sequence of numbers that converges to $0$. Moreover, $\omega^{-1}\left(\lambda\right)$ is clopen in $X$, for every $\lambda\in\C\backslash\left\{0\right\}$, and $\omega$ is constant on every component of $X$. Furthermore:
\item[(i)] If $X$ is connected and $\dim\mathbf{F}=\8$, then $\omega\equiv 0$.
\item[(ii)] If the set of point evaluations on $\mathbf{F}$ is linearly independent, then $X\backslash\omega^{-1}\left(0\right)$ is at most countable collection of isolated points in $X$.
\end{proposition}
\begin{proof}
Recall that $\omega\left(x\right)$ is an eigenvalue of $M_{\omega}^{*}$, for every $x\in X$. If $M_{\omega}$ is compact, then so is $M_{\omega}^{*}$  (see \cite[Theorem 15.3]{fhhmz}), from where  $\omega\left(X\right)$ is contained in a sequence of complex numbers that converges to $0$ (see \cite[Corollary 15.24]{fhhmz}). Hence, $\left\{\lambda\right\}$ is clopen in $\omega\left(X\right)$, for every $\lambda\in \omega\left(X\right)\backslash\left\{0\right\}$. Since $\omega$ is continuous, it follows that $\omega^{-1}\left(\lambda\right)$ is clopen in $X$.

If $Y$ is a component of $X$, and $\lambda\in \omega\left(X\right)\backslash\left\{0\right\}$, then $\omega^{-1}\left(\lambda\right)\cap Y$ is clopen in $Y$. Since the latter is connected, it follows that either $\left.\omega\right|_{Y}\equiv \lambda$, or $\omega^{-1}\left(\lambda\right)\cap Y=\varnothing$. Consequently, either $\left.\omega\right|_{Y}\equiv 0$, or $\omega^{-1}\left(0\right)\cap Y=\varnothing$. Hence, $\omega$ is constant on every component of $X$.\medskip

(i): If $X$ is connected, $\omega$ is a constant function, and so $M_{\omega}=\lambda Id_{\mathbf{F}}$, for some $\lambda\in\C$. In the case when $\dim\mathbf{F}=\8$ the only value of $\lambda$ compatible with compactness of $\lambda Id_{\mathbf{F}}$ is $0$.\medskip

(ii): Let $\lambda\in \omega\left(X\right)\backslash\left\{0\right\}$. Due to compactness, the eigenspace of $M_{\omega}^{*}$ corresponding to $\lambda$ is finitely dimensional (see \cite[Corollary 15.24]{fhhmz}). Since the point evaluations at the elements of $\omega^{-1}\left(\lambda\right)$ belong to that eigenspace and are linearly independent, it follows that $\omega^{-1}\left(\lambda\right)$ is a clopen finite set, and so consists of isolated points.
\end{proof}

In the special case when $\mathbf{F}$ has the Dunford-Pettis property every weakly compact operator on $\mathbf{F}$ is completely continuous, and so every square of a weakly compact operator is compact. Observe that $M_{\omega}^{2}=M_{\omega^{2}}$, and the conclusions of Proposition \ref{moc} hold for $\omega$ if and only if they hold for $\omega^{2}$. Therefore, we get the following result.

\begin{corollary}\label{dp}
The conclusions of Proposition \ref{moc} hold if $\mathbf{F}$ has the Dunford-Pettis property and $M_{\omega}$ is merely weakly compact.
\end{corollary}

In particular, it follows that there is no non-zero weakly compact multiplication operators on many classical function spaces that satisfy Dunford-Pettis condition, including weighted spaces of continuous functions, some of the weighted spaces of holomorphic functions, the ball and (poly)disk algebras, some of the Sobolev spaces, weighted little Lipschitz spaces, the (little) Bloch space and the Bergman and Besov spaces with exponent $1$.

While the Hardy space $H^{1}$ does not have the Dunford-Pettis property, it is known that every weakly compact weighted composition operator on it is automatically compact, from where we again can see that there are no non-zero weakly compact multiplication operators on $H^{1}$. The same property also holds for many of the Lipschitz spaces.

\section{Mean Ergodic Multiplication Operators}

In this section we will discuss possibilities for multiplication operators to be mean ergodic or uniformly mean ergodic. Our proofs are mostly inspired by the methods in \cite{bjr} and \cite{br}, where similar questions were investigated for MO's on weighted spaces of continuous and holomorphic functions. Recall that a continuous linear operator $T$ on a Banach space $E$ is called (resp \emph{uniformly}) \emph{mean ergodic} if there is $P\in\Lo\left(E\right)$ such that $\frac{1}{n}\sum\limits_{k=1}^{n}T^{k}f\xrightarrow{n\to\8} Pf$, for every $f\in E$ (resp $\frac{1}{n}\sum\limits_{1}^{n}T^{n}\xrightarrow{n\to\8} P$ in $\Lo\left(E\right)$). Also, recall that $T\in \Lo\left(E\right)$ is called power-bounded if $\sup\limits_{n\in\N}\|T^{n}\|<+\8$, and more generally Cesaro bounded if $\sup\limits_{n\in\N}\left\|\frac{1}{n}\sum\limits_{k=1}^{n}T^{k}\right\|<+\8$. Let us also recall a characterization of mean ergodicity (see e.g. \cite[Chaper 2, theorems 1.1, 1.3, 1.4 and 2.1]{krengel}).

\begin{theorem}[Mean Ergodic Theorem]
Let $T$ be a power-bounded operator on a Banach space $E$. Then the following are equivalent:
\item[(i)] $T$ is mean ergodic;
\item[(ii)] The sequence $\left\{\frac{1}{n}\sum\limits_{k=1}^{n}T^{k}f\right\}_{n\in\N}$ has a weak cluster point, for every $f\in E$;
\item[(iii)] For every $\nu\in E^{*}\backslash\left\{0_{E^{*}}\right\}$ such that $T^{*}\nu=\nu$ there is $f\in E$ such that $Tf=f$ and $\left<f,\nu\right>\ne 0$;
\item[(iv)] $E\cong \Ker \left(Id_{E}-T\right)\oplus \overline{\left(Id_{E}-T\right)E}$.\medskip

Furthermore, $T$ is uniformly mean ergodic if and only if it is mean ergodic and $\left(Id_{E}-T\right)E$ is closed.
\end{theorem}

Throughout this section $\mathbf{F}$ is a $1$-independent BSCF over a Hausdorff space $X$. For $\omega:X\to\C$ define $\omega_{n}=\frac{1}{n}\sum\limits_{k=1}^{n}\omega^{k}$. Note that $\omega_{n}\left(x\right)=1$ if $\omega\left(x\right)=1$, and $\omega_{n}\left(x\right)=\frac{\omega\left(x\right)\left(1-\omega\left(x\right)^{n}\right)}{n\left(1-\omega\left(x\right)\right)}$ otherwise. If $\|\omega\|_{\8}\le 1$ then $\omega_{n}\xrightarrow{n\to\8}\1_{\omega^{-1}\left(1\right)}$ pointwise; if $\left|\omega\left(x\right)\right|>1$, then $\left|\omega_{n}\left(x\right)\right|\xrightarrow{n\to\8}+\8$. It follows immediately from identities $M_{\omega}^{n}=M_{\omega^n}$ and $\frac{1}{n}\sum\limits_{k=1}^{n}M_{\omega}^{k}=M_{\omega_{n}}$ and Proposition \ref{hc} that if $M_{\omega}$ is a Cesaro bounded operator on a $1$-independent NSCF $\mathbf{F}$ over $X$, then $\|\omega\|_{\8}\le 1$.

Let us establish necessary conditions for MO's to be (uniformly) mean ergodic. Note that if $\omega$ is a multiplier of $\mathbf{F}$, then $\left.M_{\omega}\right|_{\mathbf{F}_{\omega^{-1}\left(1\right)}}=Id_{\mathbf{F}_{\omega^{-1}\left(1\right)}}$, and $M_{\omega}\mathbf{F}_{Y}\subset \mathbf{F}_{Y}$, for any $Y\subset X$.

\begin{theorem}\label{mem}
If $\mathbf{F}$ is $1$-independent and $\omega:X\to\C$ then:
\item[(i)] If $M_{\omega}$ is a mean ergodic operator on $\mathbf{F}$, then $\|\omega\|_{\8}\le 1$, $\omega^{-1}\left(1\right)$ is clopen, $\mathbf{F}\cong \mathbf{F}_{\omega^{-1}\left(1\right)}\oplus\mathbf{F}_{X\backslash\omega^{-1}\left(1\right)} $ and $\left.M_{1-\omega}\right|_{\mathbf{F}_{X\backslash\omega^{-1}\left(1\right)}}$ is injective and has a dense range.
\item[(ii)] If $M_{\omega}$ is uniformly mean ergodic, then additionally $\inf\limits_{x\in X\backslash \omega^{-1}\left(1\right)}\left|1-\omega\left(x\right)\right|>0$ and $\left.M_{1-\omega}\right|_{\mathbf{F}_{X\backslash\omega^{-1}\left(1\right)}}$ is a linear homeomorphism.
\end{theorem}
\begin{proof}
First, let us show $\|\omega\|_{\8}\le 1$. Assume that $\left|\omega\left(x\right)\right|>1$. Let $f\in \mathbf{F}$ be such that $f\left(x\right)\ne 0$. Then $\left|\left[\frac{1}{n}\sum\limits_{k=1}^{n}M_{\omega}^{k}f\right]\left(x\right)\right|=\left|\omega_{n}\left(x\right)f\left(x\right)\right|\xrightarrow{n\to\8}+\8$, which contradicts the fact that the sequence $\left\{\frac{1}{n}\sum\limits_{k=1}^{n}M_{\omega}^{k}f\right\}_{n\in\N}$ converges in $\mathbf{F}$.

Since $\omega$ is continuous due to Proposition \ref{hc} it follows that $\omega^{-1}\left(1\right)$ is closed. Assume that $\omega\left(x\right)=1$. Then, $M_{\omega}^{*}x_{\mathbf{F}}=x_{\mathbf{F}}$, and so from the Mean Ergodic Theorem there is $f\in \mathbf{F}$ such that $M_{\omega}f=f$ and $f\left(x\right)=\left<f,x_{\mathbf{F}}\right>\ne 0$. Hence, from Lemma \ref{fp} the set $\omega^{-1}\left(1\right)$ is open, and therefore clopen. Additionally, $\left.M_{1-\omega}\right|_{\mathbf{F}_{X\backslash\omega^{-1}\left(1\right)}}$ is injective since $1-\omega$ does not vanish on $X\backslash\omega^{-1}\left(1\right)$.

Next, Mean Ergodic Theorem implies that $\mathbf{F}$ is decomposed as $\Ker \left(Id_{\mathbf{F}}-M_{\omega}\right) \oplus \overline{\left(Id_{\mathbf{F}}-M_{\omega}\right)\mathbf{F}}$. Since $Id_{\mathbf{F}}-M_{\omega}=M_{1-\omega}$ it follows that $\Ker \left(Id_{\mathbf{F}}-M_{\omega}\right)=\mathbf{F}_{\omega^{-1}\left(1\right)}$, while $\overline{\left(Id_{\mathbf{F}}-M_{\omega}\right)\mathbf{F}}\subset \mathbf{F}_{X\backslash\omega^{-1}\left(1\right)}$. Finally, as $\mathbf{F}_{\omega^{-1}\left(1\right)}\cap \mathbf{F}_{X\backslash\omega^{-1}\left(1\right)}=\left\{0\right\}$, we conclude that $\mathbf{F}\cong \mathbf{F}_{\omega^{-1}\left(1\right)}\oplus\mathbf{F}_{X\backslash\omega^{-1}\left(1\right)} $ and $\overline{M_{1-\omega}\mathbf{F}}=\overline{M_{1-\omega}\mathbf{F}_{X\backslash\omega^{-1}\left(1\right)}}=\mathbf{F}_{X\backslash\omega^{-1}\left(1\right)}$ .\medskip

If $M_{\omega}$ is uniformly mean ergodic, it is mean ergodic, and so replacing $X$ with a clopen set $X\backslash\omega^{-1}\left(1\right)$, and $\mathbf{F}$ with $\mathbf{F}_{X\backslash\omega^{-1}\left(1\right)}$ we may assume that $Id_{\mathbf{F}}-M_{\omega}=M_{1-\omega}$ is injective on $\mathbf{F}$ with a dense range. Moreover, from the Mean Ergodic Theorem $Id_{\mathbf{F}}-M_{\omega}$ has a closed range, and so $M_{1-\omega}$ is a linear homeomorphism. As $M_{1-\omega}^{-1}=M_{\frac{1}{1-\omega}}$ it follows that $\frac{1}{1-\omega}$ is a bounded function, from where $\inf\limits_{x\in X}\left|1-\omega\left(x\right)\right|>0$.
\end{proof}

\begin{corollary}\label{cum} Assume that $X$ is connected and $\mathbf{F}$ is $1$-independent. If $\omega:X\to\C$ is a non-constant function such that $M_{\omega}$ is a (uniformly) mean ergodic operator on $\mathbf{F}$, then $\|\omega\|_{\8}\le 1$, $1\not\in \omega\left(X\right)$, ($1\not\in \overline{\omega\left(X\right)}$) and  $M_{1-\omega}$ is injective and has a dense range (is a linear homeomorphism). Furthermore, $\frac{1}{n}\sum\limits_{k=1}^{n}M_{\omega}^{k}$ converges to $0$ in the strong operator (norm) topology.
\end{corollary}

\begin{remark}\label{mem2}Note that every weakly compact power-bounded operator satisfies the condition (ii) of the Mean Ergodic theorem, and so it is mean ergodic. Hence, if $\|M_{\omega}\|\le 1$ it is mean ergodic, and so $\omega^{-1}\left(1\right)$ is clopen, from part (i) of Theorem \ref{mem}. This is an alternative proof of Theorem \ref{mp}.
\qed\end{remark}

\begin{remark}\label{mem1}The converse of Theorem \ref{mem} is the following: if $M_{\omega}$ is a power-bounded operator on $\mathbf{F}$, $\mathbf{F}\cong \mathbf{F}_{\omega^{-1}\left(1\right)}\oplus\mathbf{F}_{X\backslash\omega^{-1}\left(1\right)} $, and $\left.M_{1-\omega}\right|_{\mathbf{F}_{X\backslash\omega^{-1}\left(1\right)}}$ has a dense range (is a surjection), then $M_{\omega}$ is (uniformly) mean ergodic.
\qed\end{remark}

Note that unlike the mean ergodicity the condition of uniform mean ergodicity of an operator on $\mathbf{F}$ only depends on the properties of the operator as an element of $\Lo\left(\mathbf{F}\right)$ with no reference to $\mathbf{F}$ itself. Hence, it is plausible to expect a simple characterization of uniform mean ergodicity in the NSCF's whose multiplier algebra has an explicit topology. Let us consider a class of such NSCF's. We will say that $\mathbf{F}$ has a \emph{quasi-monotone} norm if there is $\alpha\ge 1$ such that if $\|f\|\le\alpha\|g\|$ for every $f,g\in \mathbf{F}$ with $\left|f\right|\le\left|g\right|$. It is easy to see that in this case $\|M_{\omega}\|\le \alpha\|\omega\|_{\8}$, for every multiplier $\omega$ of $\mathbf{F}$. If moreover $\|\omega\|_{\8}\le 1$, then $M_{\omega}$ is power-bounded. Another useful property is given in the following proposition.

\begin{proposition}\label{alpha}If $\mathbf{F}$ has a quasi-monotone norm, then the set of its bounded continuous multipliers is closed in $\Co_{\8}\left(X\right)$.
\end{proposition}
\begin{proof}
Let $\left\{\omega_{n}\right\}_{n\in\N}$ be a sequence of bounded continuous functions that are multipliers of $\mathbf{F}$ and let $\omega\in \Co_{\8}\left(X\right)$ be such that $\|\omega-\omega_{n}\|_{\8}\xrightarrow{n\to\8} 0$. Let $f\in\mathbf{F}$. Since the multiplier norm does not exceed $\alpha\|\cdot\|_{\8}$, for some $\alpha\ge 1$, it follows that $\left\{\omega_{n}f\right\}_{n\in\N}\subset\mathbf{F}$ is a Cauchy sequence, and so there is $g\in\mathbf{F}$ such that $\omega_{n}f\xrightarrow{n\to\8} g$ in $\mathbf{F}$. Since the latter is a NSCF it follows that $g=\omega f$, and since $f$ was chosen arbitrarily, we conclude that $\omega\in Mult\left(\mathbf{F}\right)$.
\end{proof}

Now we can state a sufficient condition for uniform mean ergodicity of a MO on a BSCF with a quasi-monotone norm.

\begin{proposition}\label{un}
Assume that $\mathbf{F}$ has a quasi-monotone norm and let $\omega$ be a continuous multiplier of $\mathbf{F}$. If $\|\omega\|_{\8}\le 1$ and $\inf\limits_{x\in X\backslash \omega^{-1}\left(1\right)}\left|1-\omega\left(x\right)\right|>0$, then $M_{\omega}$ is uniformly mean ergodic.
\end{proposition}
\begin{proof}
Let $U=\omega^{-1}\left(1\right)$, which is closed, and let $\beta=\inf\limits_{x\in X\backslash U}\left|1-\omega\left(x\right)\right|>0$. Then $U=\omega^{-1}\left(B\left(1,\beta\right)\right)$, and so $U$ is open and therefore clopen. Hence, $\1_{U}$ is continuous.

For $n\in\N$ we have $\omega_{n}\left(x\right)-\1_{U}\left(x\right)=0$ if $x\in U$, and $\omega_{n}\left(x\right)-\1_{U}\left(x\right)=\frac{\omega\left(x\right)}{1-\omega\left(x\right)}\frac{1-\omega\left(x\right)^{n}}{n}$, if $x\in X\backslash U$. In both cases $\left|\omega_{n}\left(x\right)-\1_{U}\left(x\right)\right|\le \frac{1}{\beta}\frac{2}{n}\xrightarrow{n\to\8} 0$, and so $\omega_{n}\xrightarrow{n\to\8}\1_{U}$ in $\Co_{\8}\left(X\right)$. Then, from Proposition \ref{alpha} $\1_{U}$ is a multiplier of $\mathbf{F}$. Since $\frac{1}{n}\sum\limits_{k=1}^{n}M_{\omega}^{k}=M_{\omega_{n}}$, and the norm is quasi-monotone
we conclude that $\frac{1}{n}\sum\limits_{k=1}^{n}M_{\omega}^{k}\xrightarrow{n\to\8} M_{\1_{U}}$ in $\Lo\left(\mathbf{F}\right)$. Thus, $M_{\omega}$ is uniformly mean ergodic.
\end{proof}

In the case when $X$ is compact and $\omega^{-1}\left(1\right)$ is clopen, $X\backslash \omega^{-1}\left(1\right)$ is compact, and so $\inf\limits_{x\in X\backslash \omega^{-1}\left(1\right)}\left|1-\omega\left(x\right)\right|>0$ automatically. Hence, we get the following equivalence.

\begin{corollary}\label{unc}
Assume that $X$ is compact and $\mathbf{F}$ is $1$-independent and has a quasi-monotone norm. Then MO on $\mathbf{F}$ is uniformly mean ergodic if and only if it is mean ergodic.
\end{corollary}

\begin{example}\label{una}
Assume that $X$ is a bounded domain in $\C^{n}$. Since every MO on $A\left(X\right)$ is also a MO on $A\left(\overline{X}\right)$, and the latter is a $1$-independent BSCF over a compact space $\overline{X}$ with a quasi-monotone norm, every mean ergodic MO on $A\left(X\right)$ is in fact uniformly mean ergodic. The same argument works for any uniform algebra (see \cite{gamelin}).
\qed\end{example}

\begin{remark}\label{wme}
If $X$ is connected, and $\mathbf{F}$ has a quasi-monotone norm and is such that every mean ergodic MO is uniformly mean ergodic, then there is no weakly compact MO's on $\mathbf{F}$ other than possibly $\lambda Id_{\mathbf{F}}$, $\lambda\in\C$. If $\omega$ is non-constant and such that $M_{\omega}$ is weakly compact, then $M_{\frac{\omega}{\lambda \|\omega\|_{\8}}}$ is power-bounded due to quasi-monotonicity of the norm, for every $\lambda\in\T$. Moreover, $M_{\frac{\omega}{\lambda \|\omega\|_{\8}}}$ is mean ergodic, according to Remark \ref{mem2}, and so it is uniformly mean ergodic, from our assumption. From Corollary \ref{cum} therefore $\lambda\|\omega\|_{\8}\not\in \overline{\omega\left(X\right)}$, and since $\lambda$ was chosen arbitrary, we get that  $\|\omega\|_{\8}<\|\omega\|_{\8}$. Contradiction.
\qed\end{remark}

Let us now discuss some sufficient conditions for mean ergodicity. Assume that $\|\omega\|_{\8}\le1$. We know that $\frac{1}{n}\sum\limits_{k=1}^{n}M_{\omega}^{k}f$ converges pointwise to $\1_{\omega^{-1}\left(1\right)}f$, for every $f\in \mathbf{F}$. Hence, if additionally $M_{\omega}$ is power-bounded, $\omega^{-1}\left(1\right)$ is clopen, and every pointwise convergent bounded sequence in $\mathbf{F}$ converges weakly, $M_{\omega}$ satisfies the condition (ii) of the Mean Ergodic theorem, and so is mean ergodic. Let us explore how to guarantee these conditions in a tangible way. First, recall that if $\mathbf{F}$ has a quasi-monotone norm, then $M_{\omega}$ is power-bounded. It is also easy to see that if $\mathbf{F}$ is an algebra with submultiplicative norm, the power-boundedness follows from $\|\omega\|\le 1$. Applying Grothendieck completion theorem (see \cite[III.6, Theorem 1]{bourbaki}) to $\mathbf{F}$ with pointwise topology and the collection of closed bounded sets in $\mathbf{F}$, it follows that pointwise and weak topologies coincide on bounded subsets of $\mathbf{F}$ if and only if $\mathbf{F}^{*}=\overline{\spa\left\{x_{\mathbf{F}},~x\in X\right\}}$. In turn, this happens whenever $\mathbf{F}^{**}$ is a BSCF such that $J_{\mathbf{F}^{**}}=J_{\mathbf{F}}^{**}$. Indeed, injectivity of $J_{\mathbf{F}}^{**}$ can be expressed by  $\left\{0_{\mathbf{F}^{**}}\right\}=\Ker J_{\mathbf{F}}^{**}=\left\{x_{\mathbf{F}},~x\in X\right\}^{\bot}$, which  (in the duality of $\mathbf{F}^{*}$ and $\mathbf{F}^{**}$) is equivalent to $$\mathbf{F}^{*}=\left\{0_{\mathbf{F}^{**}}\right\}^{\bot}=\left\{x_{\mathbf{F}},~x\in X\right\}^{\bot\bot}=\overline{\spa\left\{x_{\mathbf{F}},~x\in X\right\}}.$$

\begin{example}
If $v$ is continuous on $X$, the spaces $\Co^{0}_{v}$ considered in \cite{bjr} are BSCF's with quasi-monotone norms, and moreover every pointwise convergent bounded sequence is weakly convergent. Hence, $M_{\omega}$ is mean ergodic on $\Co^{0}_{v}$ as soon as $\|\omega\|_{\8}\le1$ and $\omega^{-1}\left(1\right)$ is clopen. Similarly, if $X$ is a domain in $\C^{n}$, the spaces $\Ho^{0}_{v}$ considered in \cite{br} are BSCF's with quasi-monotone norms, and moreover under some mild conditions $\left(\Ho^{0}_{v}\right)^{**}=\Ho^{\8}_{v}$ with $J_{\Ho^{\8}_{v}}=J_{\Ho^{0}_{v}}^{**}$ (see \cite{bor} and the reference therein). Hence, in this case $M_{\omega}$ is mean ergodic on $\Ho^{0}_{v}$ as soon as $\omega\in\Ho\left(X\right)$ satisfies $\|\omega\|_{\8}\le1$.
\qed\end{example}

Note that if $\mathbf{F}$ is an algebra, then every element of $\mathbf{F}$ is a multiplier of $\mathbf{F}$. It is possible to show that the set of all multipliers of a BSCF is itself a BSCF, and so from the Closed Graph theorem, $\mathbf{F}$ is continuously included into the space of its multipliers. It is easy to see that the norm $|||\cdot|||=\frac{1}{\alpha}\|\cdot\|$ on $\mathbf{F}$ is submultiplicative, where $\alpha$ be the norm of the inclusion. Using this observation and similar arguments as in the preceding example one can obtain a sufficient condition of mean ergodicity of MO's on Little Lipschitz spaces (see \cite{weaver}).\medskip

Let us conclude this article by recovering a result from \cite{bjr}. We will need the following auxiliary fact.

\begin{lemma}
Let $\mathbf{F}$ be a $1$-independent NSCF over $X$. If $\left\{x^{n}\right\}_{n\in\N}\subset X$, then $F=\left\{f\in \mathbf{F}, \lim\limits_{n\to\8}\frac{f\left(x^{n}\right)}{\|x^{n}_{\mathbf{F}}\|}=0\right\}$ is a closed subspace of $\mathbf{F}$. Moreover, if $\omega:X\to\C$ is a multiplier of $\mathbf{F}$ such that $\lim\limits_{n\to\8}\omega\left(x^{n}\right)=0$, then $M_{\omega}\mathbf{F}\subset F$.
\end{lemma}
\begin{proof}
For every $f\in \overline{F}$ and $\varepsilon>0$ there is $g\in F$ such that $\|f-g\|<\varepsilon$. Then $$\limsup\limits_{n\to\8}\frac{\left|f\left(x^{n}\right)\right|}{\|x^{n}_{\mathbf{F}}\|}\le \limsup\limits_{n\to\8}\frac{\left|g\left(x^{n}\right)\right|}{\|x^{n}_{\mathbf{F}}\|}+\limsup\limits_{n\to\8}\frac{\left|f\left(x^{n}\right)-g\left(x^{n}\right)\right|}{\|x^{n}_{\mathbf{F}}\|}\le 0+\varepsilon.$$ As $\varepsilon$ was chosen arbitrarily we conclude that $\lim\limits_{n\to\8}\frac{f\left(x^{n}\right)}{\|x^{n}_{\mathbf{F}}\|}=0$, and so $f\in F$.

To prove the second statement, take $f\in \mathbf{F}$. Then $\frac{\left|f\left(x^{n}\right)\right|}{\|x^{n}_{\mathbf{F}}\|}\le \|f\|$, for every $n\in\N$, from where $\lim\limits_{n\to\8}\frac{\left[M_{\omega}f\right]\left(x^{n}\right)}{\|x^{n}_{\mathbf{F}}\|}=\lim\limits_{n\to\8}\frac{\omega\left(x^{n}\right)f\left(x^{n}\right)}{\|x^{n}_{\mathbf{F}}\|}=0$.
\end{proof}

\begin{proposition}
Assume that $X$ is Tychonoff and $u:X\to\left(0,+\8\right)$ is upper semi-continuous. A multiplication operator is uniformly mean ergodic on $\Co^{\8}_{u}$ if and only if it is mean ergodic.
\end{proposition}
\begin{proof}
First, let us show that if in the statement of the lemma $\left\{x^{n}\right\}_{n\in\N}$ is a discrete subset of $X$ (i.e. discrete in the induced topology), and $\mathbf{F}=\Co^{\8}_{u}$, then the $F\ne \Co^{\8}_{u}$.

We will construct a sequence $\left\{V_{n}\right\}_{n\in\N}$ of open sets, such that $x^{n}\in V_{n}$, for every $n\in\N$, and $\overline{V_{j}}\cap\overline{V_{k}}=\varnothing$, for distinct $j,k\in\N$. Assume that $V_{1},...,V_{n}$ are chosen (if $n=0$ nothing is chosen yet). Since $\left\{x^{n}\right\}_{n\in\N}$ is discrete, there is an open set $U_{n+1}$ such that $U_{n+1}\cap \left\{x^{n}\right\}_{n\in\N}=\left\{x^{n+1}\right\}$. Since $x^{n+1}$ belongs to $U_{n+1}$, but none of $\overline{V_{1}},...,\overline{V_{n}}$, it follows that $W_{n+1}=U_{n+1}\backslash \bigcup\limits_{k=1}^{n}\overline{V_{k}}$ is an open neighborhood of $x^{n+1}$. Hence, choose $V_{n+1}$ to be an open neighborhood of $x^{n+1}$ such that $\overline{V_{n+1}}\subset W_{n+1}$.\medskip

Let $w_{n}:X\to\left[0,+\8\right]$ be defined as $\frac{1}{u}$ on $V_{n}$ and as $0$ on $X\backslash V_{n}$. Since $u$ is upper semi-continuous and non-vanishing it follows that $w_{n}$ is lower semi-continuous on $X$. Hence, there is $f_{n}\in \Co_{\8}\left(X\right)$ such that $\left|f_{n}\right|\le w_{n}$ and $f_{n}\left(x^{n}\right)>\frac{1}{2} w_{n}\left(x^{n}\right)$. As $f_{n}$ vanishes outside of  $V_{n}$, and $\left\{\overline{V_{n}}\right\}_{n\in\N}$ are disjoint, the sum $f=\sum\limits_{n\in\N}f_{n}$ is well-defined and continuous. Moreover, $\left|f\right|\le \frac{1}{u}$ and $f\left(x^{n}\right)=f_{n}\left(x^{n}\right)>\frac{1}{2 u\left(x^{n}\right)}$. Hence, $\limsup\limits_{n\to\8}\frac{\left|f\left(x^{n}\right)\right|}{\|x^{n}_{\Co^{\8}_{u}}\|}=\limsup\limits_{n\to\8}\left|f\left(x^{n}\right)u\left(x^{n}\right)\right|\ge\frac{1}{2}$, and so $f\in\Co^{\8}_{u}\backslash F$.\medskip

Assume that $M_{\omega}$ is mean ergodic. Replacing $X$ with a clopen set $X\backslash\omega^{-1}\left(1\right)$ if needed, without loss of generality $\upsilon=1-\omega$ does not vanish. Assume that $\inf\limits_{x\in X}\left|\upsilon\left(x\right)\right|=0$ and construct the sequence $\left\{x^{n}\right\}_{n\in\N}\subset X$ as follows. Fix arbitrary $x^{1}$, and if $x^{1},...,x^{n}$ are already constructed, choose $x^{n+1}$ to be such that $\left|\upsilon\left(x^{n+1}\right)\right|<\frac{1}{2}\left|\upsilon\left(x^{n}\right)\right|$. It is easy to see that $\left\{x^{n}\right\}_{n\in\N}$ is a discrete subset. Hence, $M_{\upsilon}\Co^{\8}_{u}\subset\left\{f\in \Co^{\8}_{u}, \lim\limits_{n\to\8}f\left(x^{n}\right)u\left(x^{n}\right)=0\right\}$, which is a closed proper subspace of $\Co^{\8}_{u}$.

Thus, $M_{1-\omega}$ does not have a dense range, which contradicts mean ergodicity of $M_{\omega}$, according to Theorem \ref{mem}. Therefore, $\inf\limits_{x\in X}\left|1-\omega\left(x\right)\right|>0$, and so $M_{1-\omega}$ is a linear homeomorphism, from where $M_{\omega}$ is uniformly mean ergodic, according to Remark \ref{mem1}.
\end{proof}

In order to recover the corresponding result from \cite{br} the following additional piece of information is required.

\begin{question}
If $X$ is a domain in $\C^{n}$, which upper semi-continuous $u:X\to\left(0,+\8\right)$ have the property that for every discrete sequence $\left\{x^{n}\right\}_{n\in\N}\subset X$ there is $f\in \Ho_{u}^{\8}$ such that $\limsup\limits_{n\to\8}\frac{\left|f\left(x^{n}\right)\right|}{\|x^{n}_{\Ho_{u}^{\8}}\|}>0$, or more specifically $\limsup\limits_{n\to\8}\left|f\left(x^{n}\right)\right|u\left(x^{n}\right)>0$?
\end{question}

Note that the required condition is satisfied if every discrete sequence contains an interpolating sequence for $\Ho_{u}^{\8}$. Some characterizations of such sequences see e.g. in \cite{dl}.


\begin{bibsection}
\begin{biblist}

\bib{bbt}{article}{
   author={Bierstedt, Klaus D.},
   author={Bonet, Jos{\'e}},
   author={Taskinen, Jari},
   title={Associated weights and spaces of holomorphic functions},
   journal={Studia Math.},
   volume={127},
   date={1998},
   number={2},
   pages={137--168},
}

\bib{erz}{article}{
   author={Bilokopytov, Eugene},
   title={Continuity and Holomorphicity of Symbols of Weighted Composition
   Operators},
   journal={Complex Anal. Oper. Theory},
   volume={13},
   date={2019},
   number={3},
   pages={1441--1464},
}

\bib{erz2}{article}{
   author={Bilokopytov, Eugene},
   title={Which multiplication operators are surjective isometries?},
   journal={J. Math. Anal. Appl.},
   volume={480},
   date={2019},
   number={1},
}

\bib{bjr}{article}{
   author={Bonet, Jos\'{e}},
   author={Jord\'{a}, Enrique},
   author={Rodr\'{\i}guez, Alberto},
   title={Mean ergodic multiplication operators on weighted spaces of
   continuous functions},
   journal={Mediterr. J. Math.},
   volume={15},
   date={2018},
   number={3},
   pages={Art. 108, 11},
}

\bib{br}{article}{
   author={Bonet, Jos\'{e}},
   author={Ricker, Werner J.},
   title={Mean ergodicity of multiplication operators in weighted spaces of
   holomorphic functions},
   journal={Arch. Math. (Basel)},
   volume={92},
   date={2009},
   number={5},
   pages={428--437},
}

\bib{bourb}{book}{
   author={Bourbaki, Nicolas},
   title={General topology. Chapters 5--10},
   series={Elements of Mathematics (Berlin)},
   note={Translated from the French;
   Reprint of the 1989 English translation},
   publisher={Springer-Verlag, Berlin},
   date={1998},
   pages={iv+363},
}

\bib{bourbaki}{book}{
   author={Bourbaki, Nicolas},
   title={Topological vector spaces. Chapters 1--5},
   series={Elements of Mathematics (Berlin)},
   note={Translated from the French by H. G. Eggleston and S. Madan},
   publisher={Springer-Verlag, Berlin},
   date={1987},
   pages={viii+364},
}

\bib{bor}{article}{
   author={Boyd, Christopher},
   author={Rueda, Pilar},
   title={The biduality problem and M-ideals in weighted spaces of
   holomorphic functions},
   journal={J. Convex Anal.},
   volume={18},
   date={2011},
   number={4},
   pages={1065--1074},
}

\bib{dl}{article}{
   author={Doma\'{n}ski, Pawe\l },
   author={Lindstr\"{o}m, Mikael},
   title={Sets of interpolation and sampling for weighted Banach spaces of
   holomorphic functions},
   journal={Ann. Polon. Math.},
   volume={79},
   date={2002},
   number={3},
   pages={233--264},
   issn={0066-2216},
   review={\MR{1957801}},
   doi={10.4064/ap79-3-3},
}

\bib{ds}{book}{
   author={Dunford, Nelson},
   author={Schwartz, Jacob T.},
   title={Linear Operators. I. General Theory},
   series={With the assistance of W. G. Bade and R. G. Bartle. Pure and
   Applied Mathematics, Vol. 7},
   publisher={Interscience Publishers, Inc., New York; Interscience
   Publishers, Ltd., London},
   date={1958},
   pages={xiv+858},
}

\bib{fhhmz}{book}{
   author={Fabian, Mari\'an},
   author={Habala, Petr},
   author={H\'ajek, Petr},
   author={Montesinos, Vicente},
   author={Zizler, V\'aclav},
   title={Banach space theory},
   series={CMS Books in Mathematics/Ouvrages de Math\'ematiques de la SMC},
   note={The basis for linear and nonlinear analysis},
   publisher={Springer, New York},
   date={2011},
   pages={xiv+820},
}

\bib{floret}{book}{
   author={Floret, Klaus},
   title={Weakly compact sets},
   series={Lecture Notes in Mathematics},
   volume={801},
   note={Lectures held at S.U.N.Y., Buffalo, in Spring 1978},
   publisher={Springer, Berlin},
   date={1980},
   pages={vii+123},
}

\bib{ggm}{article}{
   author={Galanopoulos, Petros},
   author={Girela, Daniel},
   author={Mart\'{\i}n, Mar\'{\i}a J.},
   title={Besov spaces, multipliers and univalent functions},
   journal={Complex Anal. Oper. Theory},
   volume={7},
   date={2013},
   number={4},
   pages={1081--1116},
}

\bib{gamelin}{book}{
   author={Gamelin, Theodore W.},
   title={Uniform algebras},
   publisher={Prentice-Hall, Inc., Englewood Cliffs, N. J.},
   date={1969},
   pages={xiii+257},
}

\bib{koosis}{book}{
   author={Koosis, Paul},
   title={Introduction to $H_p$ spaces},
   series={Cambridge Tracts in Mathematics},
   volume={115},
   edition={2},
   note={With two appendices by V. P. Havin [Viktor Petrovich Khavin]},
   publisher={Cambridge University Press, Cambridge},
   date={1998},
   pages={xiv+289},
}

\bib{krengel}{book}{
   author={Krengel, Ulrich},
   title={Ergodic theorems},
   series={De Gruyter Studies in Mathematics},
   volume={6},
   note={With a supplement by Antoine Brunel},
   publisher={Walter de Gruyter \& Co., Berlin},
   date={1985},
   pages={viii+357},
}

\bib{romanov}{article}{
   author={Romanov, A. V.},
   title={Wold decomposition in Banach spaces},
   language={Russian, with Russian summary},
   journal={Mat. Zametki},
   volume={82},
   date={2007},
   number={6},
   pages={894--904},
   translation={
      journal={Math. Notes},
      volume={82},
      date={2007},
      number={5-6},
      pages={806--815},
   },
}

\bib{scheidemann}{book}{
   author={Scheidemann, Volker},
   title={Introduction to complex analysis in several variables},
   publisher={Birkh\"auser Verlag, Basel},
   date={2005},
   pages={viii+171},
}

\bib{vukotic}{article}{
   author={Vukoti\'{c}, Dragan},
   title={Pointwise multiplication operators between Bergman spaces on
   simply connected domains},
   journal={Indiana Univ. Math. J.},
   volume={48},
   date={1999},
   number={3},
   pages={793--803},
}

\bib{weaver}{book}{
   author={Weaver, Nik},
   title={Lipschitz algebras. Second edition},
   publisher={World Scientific Publishing Co. Pte. Ltd., Hackensack, NJ},
   date={2018},
   pages={xiv+458},
}


\end{biblist}
\end{bibsection}
\end{document}